\documentclass[a4paper, 11pt]{amsart}
\usepackage[
  margin=1.9cm,
  includefoot,
  footskip=30pt,
]{geometry}
\usepackage{amsmath,amssymb,amscd}
\usepackage{amsfonts}
\usepackage[english]{babel}
\usepackage[leqno]{amsmath}
\usepackage{amssymb,amsthm}
\usepackage{mathrsfs} 
\usepackage{amscd}
\usepackage{enumerate}
\usepackage{epsfig}
\usepackage{relsize}
\usepackage{layout}
\usepackage[usenames,dvipsnames]{xcolor}
\usepackage[backref=page]{hyperref}
\usepackage{tikz}
\hypersetup{
 colorlinks,
 citecolor=Green,
 linkcolor=Red,
 urlcolor=Blue}
\usepackage[matrix,arrow,tips,curve]{xy}
\input{xy}
\xyoption{all}
\definecolor{darkgreen}{rgb}{0.0, 0.7, 0.0}
\definecolor{purple}{rgb}{0.5, 0.0, 0.5}
\definecolor{red}{rgb}{0.8, 0.2, 0.0}

\newtheorem{thm}{Theorem}[section]
\newtheorem{bthm}{Theorem}
\newtheorem{bcor}{Corollary}
\newtheorem{lemma}[thm]{Lemma}

\newtheorem{cor}[thm]{Corollary}
\newtheorem{claim}[thm]{Claim}

\numberwithin{equation}{section}
\theoremstyle{definition}
\newtheorem{defi}[thm]{Definition}

\theoremstyle{remark}
\newtheorem{remark}[thm]{Remark}

\newcommand{\C}{\mathbb{C}}

\def \Im{{\rm Im}}

\def \P{\mathbb{P}}

\def \ZZ{\mathbb{Z}}

\def \F{\mathcal F}

\def\I{\mathcal I}
\def \L{\mathcal L}
\def \E{\mathcal E}
\def \G{\mathcal G}
\def \H{\mathcal H}
\def \U{\mathcal U}
\def\O{\mathcal O}

\def\M0{\mathcal M^0}

\DeclareMathOperator{\Ker}{{Ker}}

\newcommand{\rk}{\operatorname{rk}}

\title[A geometrical view of Ulrich vector bundles]{A geometrical view of Ulrich vector bundles}

\author[A.F. Lopez and J.C. Sierra]{Angelo Felice Lopez* and Jos\'e Carlos Sierra}

\address{\hskip -.43cm Angelo Felice Lopez, Dipartimento di Matematica e Fisica, Universit\`a di Roma
Tre, Largo San Leonardo Murialdo 1, 00146, Roma, Italy. e-mail {\tt lopez@mat.uniroma3.it}}
\address{\hskip -.43cm Jos\'e Carlos Sierra, Departamento de Matem\'aticas Fundamentales, Facultad de Ciencias, UNED,  C/ Juan del Rosal 10, 28040 Madrid, Spain. e-mail {\tt jcsierra@mat.uned.es}}

\thanks{* Research partially supported by  PRIN ``Advances in Moduli Theory and Birational Classification'' and GNSAGA-INdAM}


\begin{document} 

\begin{abstract} 
We study geometrical properties of an Ulrich vector bundle $\E$ of rank $r$ on a smooth $n$-dimensional variety $X  \subseteq \P^N$. We characterize ampleness of $\E$ and of $\det \E$ in terms of the restriction to lines contained in $X$. We prove that all fibers of the map $\Phi_{\E}:X \to {\mathbb G}(r-1, \P H^0(\E))$ are linear spaces, as well as the projection on $X$ of all fibers of the map $\varphi_{\E} : \P(\E) \to \P H^0(\E)$. Then we get 
a number of consequences: a characterization of bigness of $\E$ and of $\det \E$ in terms of the maps $\Phi_{\E}$ and $\varphi_{\E}$; when $\det\E$ is big and $\E$ is not big there are infinitely many linear spaces in $X$ through any point of $X$; when $\det \E$ is not big, the fibers of $\Phi_{\E}$ and $\varphi_{\E}$ have the same dimension; a classification of Ulrich vector bundles whose determinant has numerical dimension at most $\frac{n}{2}$; a classification of Ulrich vector bundles with $\det \E$ of numerical dimension at most $k$ on a linear $\P^k$-bundle.
\end{abstract}

\maketitle

\section{Introduction}

Let $X$ be a smooth irreducible variety of dimension $n \ge 1$. Introducing a rank $r$ globally generated vector bundle $\E$ on $X$ gives, as is well known, two ways of investigating the geometry of $X$. One is via the map
$$\Phi_{\E} : X \to {\mathbb G}(r-1, \P H^0(\E))$$
and the other one via the map
$$\varphi_{\E} = \varphi_{\O_{\P(\E)}(1)} : \P(\E) \to \P H^0(\O_{\P(\E)}(1)) \cong \P H^0(\E).$$
The latter map together with the ones associated to multiples $\O_{\P(\E)}(m)$ measure the positivity of $\E$, for example $\E$ is ample if and only if $\varphi_{\O_{\P(\E)}(m)}$ is an embedding for $m \gg 0$ and $\E$ is big if and only if the image of 
$\varphi_{\E}$ has dimension $n+r-1$. On the other hand, considering the map
$$\lambda_{\E} : \Lambda^r H^0(\E) \to H^0(\det \E)$$
one gets a commutative diagram (see for example \cite[\S 3]{mu})
\begin{equation}
\label{muk}
\xymatrix{X \ar[d]^{\varphi_{|\Im \lambda_{\E}|}} \ar[r]^{\hskip -.9cm \Phi_{\E}} &  {\mathbb G}(r-1,\P H^0(\E)) \ar@{^{(}->}[d]^{P_{\E}} \\ \P \Im \lambda_{\E} \ar@{^{(}->}[r] & \P \Lambda^r H^0(\E)}
\end{equation}
where $P_{\E}$ is the Pl\"ucker embedding. Thus the positivity of $\det \E$ can be related to the map $\Phi_{\E}$.

There is a very interesting class of vector bundles that has recently attracted a lot of attention, namely that of Ulrich vector bundles. Recall that if $X  \subseteq \P^N$, a vector bundle $\E$ on $X$ is an Ulrich vector bundle if $H^i(\E(-p))=0$ for all $i \ge 0$ and $1 \le p \le n$. 

Ulrich vector bundles are very interesting for several reasons as they play a relevant role both in commutative algebra and in algebraic geometry (see for example \cite{es, b1, cmp} and references therein). Important reasons being that having an Ulrich vector bundle, conjecturally true in all cases, has several consequences, such has determinantal representation, Chow forms, Boij-S\"oderberg theory and so on.

Now an Ulrich vector bundle $\E$ is globally generated, hence one naturally wonders about the positivity properties of $\E$, of $\det \E$ and the behaviour of the above maps $\Phi_{\E}$ and $\varphi_{\E}$. 

It is the purpose of this paper to highlight several nice geometric features of these maps and their relation to the positivity of $\E$. 

The first main result is a characterization of Ulrich vector bundles that are ample, or that have ample 
determinant, in terms of the above maps and of lines contained in $X$, as follows. 

\begin{bthm}
\label{main1}

\hskip 3cm

Let $X  \subseteq \P^N$ be a smooth irreducible variety and let $\E$ be an Ulrich vector bundle on $X$. Then the following are equivalent:
\begin{itemize}
\item [(i)] $\E$ is very ample (that is $\varphi_{\E}$ is an embedding);
\item [(ii)] $\E$ is ample;
\item [(iii)] either $X$ does not contain lines or $\E_{|L}$ is ample on any line $L \subset X$.
\end{itemize}
Also the following are equivalent:
\begin{itemize}
\item [(iv)] $\Phi_{\E}$ is an embedding;
\item [(v)] $\det \E$ is very ample;
\item [(vi)] $\det \E$ is ample;
\item [(vii)] either $X$ does not contain lines or $\E_{|L}$ is not trivial on any line $L \subset X$.
\end{itemize}
\end{bthm}

As far as we know the best result previously known, is that if $X  \subseteq \P^N$ is not covered by lines, then any Ulrich vector bundle on $X$ is big and that if $X$ does not contain any line, then $\det \E$ is ample \cite[Thm.~1]{lo}. Thus Theorem \ref{main1} improves and clarifies \cite[Thm.~1]{lo}. Note that one cannot hope to detect bigness simply by restricting to lines (cf. Theorem \ref{main1} and Corollary \ref{cor1} below). 
For example the spinor bundle on the quadric $3$-fold is Ulrich not big, while its restriction to the hyperplane section is Ulrich and  big. But they both restrict to lines as $\O_{\P^1} \oplus \O_{\P^1}(1)$. Moreover note that varieties not containing lines are also characterized by the ampleness of a vector bundle, namely of $\Omega_X(2)$ \cite[Prop.~4.2]{br}. 

Another very nice property of Ulrich vector bundles is that the fibers of $\Phi_{\E}$ are linear spaces and the fibers of $\varphi_{\E}$ project on $X$ to linear spaces, as we show in the following result. 

We set $\nu(\E) := \nu(\O_{\P(\E)}(1))$ for the numerical dimension. 

\begin{bthm}
\label{main2} 

\hskip 3cm

Let $X  \subseteq \P^N$ be a smooth irreducible variety of dimension $n \ge 1$ and degree $d$. Let $\E$ be a rank $r$ Ulrich vector bundle on $X$. Then any 
scheme-theoretic 
fiber of $\Phi_{\E} : X \to {\mathbb G}(r-1,rd-1)$ is a linear subspace in $\P^N$ of dimension at least $n - \nu(\det \E)$. 

Moreover any 
scheme-theoretic 
fiber of $\varphi_{\E}: \P(\E) \to \P^{rd-1}$ projects isomorphically onto a linear subspace in $\P^N$ of dimension at least $n+r-1-\nu(\E)$ contained in $X$.
\end{bthm}


It is clear that, unless the image 
has dimension $1$, the fibers of $\Phi_{\E}$ and $\varphi_{\E}$ can actually have different dimensions, see for example Remark \ref{delp}.

The above theorems and the methods used to prove them lead to several consequences in some interesting cases, 
listed below. See also Remark \ref{kgg} for global generation of adjoint bundles. 


Notice that for any globally generated vector bundle $\E$ one has that if $\varphi_{\E}$ (respectively $\Phi_{\E}$) is birational onto its image, then $\E$ (respectively $\det \E$) is big. It is a nice fact that for an Ulrich vector bundle the converse holds.

\begin{bcor}
\label{cor1}
\hskip 3cm

Let $X  \subseteq \P^N$ be a smooth irreducible variety and let $\E$ be an Ulrich vector bundle on $X$. Then the following are equivalent:
\begin{itemize}
\item [(i)] $\varphi_{\E}$ is birational onto its image;
\item [(ii)] $\E$ is big.
\end{itemize}
Similarly, the following are equivalent:
\begin{itemize}
\item [(iii)] $\Phi_{\E}$ is birational onto its image;
\item [(iv)] $\det \E$ is big.
\end{itemize}
\end{bcor}


Next, when $\det \E$ is big and $\E$ is not big, we show that there are actually infinitely many linear subspaces of dimension $n+r-1-\nu(\E)$ contained in $X$ and passing through any point. 

We let $\phi(\E)$ be the dimension of the general fiber of $\Phi_{\E}$.

\begin{bcor}
\label{cor2}

\hskip 3cm

Let $X  \subseteq \P^N$ be a smooth irreducible variety of dimension $n \ge 1$ and let $\E$ be a rank $r$ Ulrich vector bundle on $X$ such that $c_1(\E)^n>0$ (or, equivalently, $\phi(\E) = 0$). If $\E$ is not big then through any point of $X$ there are infinitely many $(n+r-1-\nu(\E))$-dimensional linear subspaces of $\P^N$ contained in $X$.
\end{bcor}


On the other hand, we can draw several consequences when $\det \E$ is not big or has lower numerical dimension.

First we have a result linking the fibers of $\Phi_{\E}$ and $\varphi_{\E}$ when $\det \E$ is not big (that is when $c_1(\E)^n=0$).

\begin{bcor} 
\label{cor3}

\hskip 3cm

Let $X  \subseteq \P^N$ be a smooth irreducible variety of dimension $n \ge 1$ and let $\E$ be a rank $r$ Ulrich vector bundle on $X$ such that $c_1(\E)^n=0$ (or, equivalently, $\phi(\E) \ge 1$). Then 
$$\nu(\E)+\phi(\E)=n+r-1$$ 
that is the general fibers of $\Phi_{\E}$ and $\varphi_{\E}$ have the same dimension. Moreover $X\subseteq \P^N$ is covered by a family of linear $\P^{\phi(\E)}$'s.
\end{bcor}


We also get a classification of Ulrich vector bundles whose determinant has numerical dimension at most $\frac{n}{2}$.

\begin{bcor}
\label{cor4}

\hskip 3cm

Let $X  \subseteq \P^N$ be a smooth irreducible variety of dimension $n \ge 1$. Let $\E$ be a rank $r$ vector bundle on $X$. Then $\E$ is Ulrich with $c_1(\E)^{\lfloor \frac{n}{2} \rfloor +1} =0$ if and only if $(X,\O_X(1), \E)$ is one of the following:
\begin{itemize}
\item [(i)] $(\P^n, \O_{\P^n}(1), \O_{\P^n}^{\oplus r})$;
\item [(ii)] $(\P(\F), \O_{\P(\F)}(1), p^*(\G(\det \F)))$, where $\F$ is a very ample rank $n-b+1$ vector bundle on a smooth irreducible projective variety $B$ of dimension $b$ with $1 \le b \le \frac{n}{2}$, $p: X \cong \P(\F) \to B$ is the projection and $\G$ is a rank $r$ vector bundle on $B$ such that $H^q(\G \otimes S^k \F^*)=0$ for $q \ge 0, 0 \le k \le b-1$.
\end{itemize}
\end{bcor}

%
%
%
%
%

The above result improves \cite[Thm.~3]{lo}, that was proved only in dimension $3$ and with several exceptions. 
Note that the power $\lfloor \frac{n}{2} \rfloor +1$ in Corollary \ref{cor4} is sharp, see Remark \ref{sha}. 


In the above corollary naturally appears the case when $X$ is a linear $\P^k$-bundle over a smooth variety. In fact, in this case, we can classify Ulrich vector bundles when the numerical dimension of its determinant does not exceed $k$.

\begin{bcor}
\label{cor5}

\hskip 3cm

Let $B$ be a smooth irreducible variety of dimension $b \ge 1$ and let $\F$ be a very ample rank $n-b+1 \ge 1$ vector bundle on $B$. Let $X = \P(\F) \subset \P^N$ be embedded with the tautological line bundle $H$ and let $p: X \to B$ be the projection. Let $\E$ be a rank $r$ vector bundle on $X$. 

Then $\E$ is Ulrich with $c_1(\E)^{n-b+1}=0$ if and only if $(X,H,\E)$ is as follows:
\begin{itemize}
\item[(i)] $(\P^b \times \P^{n-b}, \O_{\P^b}(1) \boxtimes \O_{\P^{n-b}}(1), q^*(\O_{\P^{n-b}}(b))^{\oplus r})$, where $q:X \to \P^{n-b}$ is the second projection;
\item[(ii)] $(\P(\F), \O_{\P(\F)}(1), p^*(\G(\det \F)))$, where $\G$ is a rank $r$ vector bundle on $B$ such that $H^q(\G \otimes S^k \F^*)=0$ for $q \ge 0$ and $0 \le k \le b-1$. Moreover $b \le \frac{n}{2}$ and $c_1(\G(\det \F))^b \ne 0$.
\end{itemize} 
\end{bcor}
Finally, in the case of $\P^b \times \P^{n-b}$, a more precise result can be proved, see Corollary \ref{prodotto}.

\section{Notation and standard facts about Ulrich vector bundles}

We collect in this section some useful definitions, notation and facts. 

Given a nef line bundle $\L$ on a smooth variety $X$ we denote by
$$\nu(\L) = \max\{k \ge 0: \L^k \ne 0\}$$ 
the numerical dimension of $\L$.

\begin{defi}
Let $X$ be a smooth irreducible variety and let $\E$ be a vector bundle on $X$. We say that $\E$ is {\it nef (big, ample, very ample)} if $\O_{\P(\E)}(1)$ is nef (big, ample, very ample). If $\E$ is nef, we define the numerical dimension of $\E$ by
$\nu(\E):=  \nu(\O_{\P(\E)}(1))$. We will denote by $\pi : \P(\E) \to X$ the projection map.
\end{defi}

\begin{defi}
Let $X$ be a smooth irreducible variety of dimension $n \ge 1$ and let $\E$ be a globally generated rank $r$ vector bundle on $X$. We set $\phi(\E)$ to be the dimension of the general fiber of $\Phi_{\E} : X \to {\mathbb G}(r-1, \P H^0(\E))$. For every $x \in X$ we set $P_x = \varphi_{\E}(\P(\E_x))$. 
\end{defi}
Note that diagram \eqref{muk} implies that
\begin{equation}
\label{pc}
\phi(\E)=n-\nu(\det \E).
\end{equation}
Moreover note that the point $\Phi(x) \in {\mathbb G}(r-1,\P H^0(\E))$ corresponds to the linear subspace $P_x \subseteq \P H^0(\E)$. We will write this as
$$\Phi(x)= [P_x]$$
in order to distinguish the subspace  $P_x \subseteq \P H^0(\E)$ with the point $[P_x] \in  {\mathbb G}(r-1,\P H^0(\E))$. 

A simple fact that will be useful is the following.

\begin{remark}
\label{inj}
For every $x \in X$ the restriction morphism
$$\varphi_{|\P(\E_x)} : \P(\E_x) \to P_x$$
is an isomorphism onto a linear subspace of dimension $r-1$ in $\P^{rd-1}$. In particular 
$\pi_{|F} : F \to X$ is a closed embedding on any scheme theoretic fiber $F$ of $\varphi$. 
\end{remark}

We now turn to a few generalities on Ulrich vector bundles. 
\begin{defi}
%

Let $X \subseteq \P^N$ be a smooth irreducible variety of dimension $n \ge 1$ and let $\E$ be a vector bundle on $X$. We say that $\E$ is an {\it Ulrich vector bundle} if $H^i(\E(-p))=0$ for all $i \ge 0$ and $1 \le p \le n$.
\end{defi}

Observe that if $X \subseteq \P^N$ is a smooth irreducible variety of degree $d$ and $\E$ is a rank $r$ Ulrich vector bundle on $X$, then $\E$ is $0$-regular in the sense of Castelnuovo-Mumford. Hence $\E$  is globally generated (see for instance \cite[Thm.~1.8.5]{laz}). Moreover $\E$ is ACM and $h^0(\E)=rd$ (see for example \cite[Prop.~2.1]{es} or \cite[(3.1)]{b1}). These facts will be often used without further reference.

\section{Separation lemmas}

We start with some preliminary results, the goal being understanding the separation properties of the linear systems $|\O_{\P(\E)}(1)|$ and $|\Im \lambda_{\E}|$ in the case $\E$ is an Ulrich vector bundle.


\begin{lemma}
\label{rest}

Let $X  \subseteq \P^N$ be a smooth irreducible variety of dimension $n \ge 1$. Let $\E$ be a rank $r$ Ulrich vector bundle on $X$. Let $s \in \{1, \ldots, n\}$ and let $H_1, \dots, H_s$ be hyperplane sections of $X$. Let $X_s = H_1 \cap \dots \cap H_s$ and suppose that $X_s$ is of 
dimension $n-s$. 

Then $H^i(\E_{|X_s}(-p))=0$ for all $i \ge 0$ and $1 \le p \le n-s$. Moreover
$$H^0(\E) \to H^0(\E_{|X_s})$$
is an isomorphism.
\end{lemma}
\begin{proof}
For $0 \le j \le s$ set $X_0=X$ and $X_j = H_1 \cap \dots \cap H_j$. We prove that, for any $0 \le j \le s$, the following hold:
\begin{equation}
\label{ul} 
H^i(\E_{|X_j}(-p))=0 \ \hbox{for all} \ i \ge 0 \ \hbox{and} \ 1 \le p \le n-j,
\end{equation}
and that
\begin{equation}
\label{iso} 
H^0(\E) \to H^0(\E_{|X_j}) \ \hbox{is an isomorphism.}
\end{equation}

If $j=0$ note that \eqref{ul} is just the fact that $\E$ is Ulrich and \eqref{iso} is obvious. If $j \ge 1$ assume that \eqref{ul} and \eqref{iso} hold for $j-1$ and consider the exact sequence
$$0 \to \E_{|X_{j-1}}(-p-1) \to \E_{|X_{j-1}}(-p) \to \E_{|X_j}(-p) \to 0.$$
Then clearly \eqref{ul} holds for $j$ and this also gives that $H^0(\E_{|X_{j-1}}) \to H^0(\E_{|X_j})$ is an isomorphism. Therefore so is $H^0(\E) \to H^0(\E_{|X_j})$. This proves \eqref{iso} for $j$ and the lemma.
\end{proof}

We now apply the previous lemma to study the injectivity of $\varphi_{\E}$ and $\Phi_{\E}$ and of their differentials.

First we study the case when two points, possibly infinitely near, do not belong to a line contained in $X$.

\begin{lemma}
\label{span}

Let $X  \subseteq \P^N$ be a smooth irreducible variety of dimension $n \ge 1$ and degree $d$. Let $\E$ be a rank $r$ Ulrich vector bundle on $X$. Let $Z \subset X$ be a $0$-dimensional subscheme of length $2$ and let $L \subset \P^N$ be the line generated by $Z$ and suppose that $L \not\subset X$. Then
$$H^0(\E) \to H^0(\E_{|Z})$$
is surjective. Now assume that $Z=\{x,x'\}$ with $x \ne x'$. Then 
$$P_x \cap P_{x'} = \emptyset$$ 
hence, in particular
$$\Phi_{\E}(x) \ne \Phi_{\E}(x').$$
\end{lemma}
\begin{proof}
Let $H_1, \dots, H_n$ be general hyperplane sections containing $Z$. Then $X_n = H_1 \cap \dots \cap H_n$ is a $0$-dimensional subscheme of $X$ of length $d$ containing $Z$, 
as $L\not\subset X$. 
By Lemma \ref{rest} we have that $H^0(\E) \to H^0(\E_{|X_n})$ is surjective, hence so is $H^0(\E) \to H^0(\E_{|Z})$. Therefore $h^0(\I_{Z/X} \otimes \E)=h^0(\E)-2r=r(d-2)$. If $Z=\{x,x'\}$ with $x \ne x'$, considering the isomorphism
$$H^0(\I_{\P(\E_x) \cup \P(\E_{x'})/\P(\E)} \otimes \O_{\P(\E)}(1)) \cong H^0(\I_{\{x,x'\}/X} \otimes \E)$$
we deduce that $\dim \langle P_x, P_{x'} \rangle = 2r-1$, hence that $P_x \cap P_{x'} = \emptyset$. 
\end{proof}

In the case two points, possibly infinitely near, belong to a line contained in $X$, we will use the following.

\begin{lemma}
\label{rette}

Let $X  \subseteq \P^N$ be a smooth irreducible variety of dimension $n \ge 1$. Let $\E$ be a rank $r$ Ulrich vector bundle on $X$. Let $L \subseteq X$ be a line. Then 
$$H^0(\E) \to H^0(\E_{|L})$$
is surjective. Moreover either $(\Phi_{\E})_{|L}$ is an isomorphism onto its image or it is constant and, in that case, $\E_{|L}$ is trivial. 
\end{lemma}
\begin{proof}
The lemma being obvious if $n=1$, we suppose that $n \ge 2$ and let $H_1, \dots, H_n$ be general hyperplane sections (of $X$) containing $L$. By Lemma \ref{rest} we have that $H^1(\E_{|X_{n-1}}(-1))=0$ and the map 
$$H^0(\E) \to H^0(\E_{|X_{n-1}})$$
is an isomorphism. Now $X_{n-1} = C \cup L$, where $C$ is a curve not containing $L$ or $C= \emptyset$, and in the latter case we are done. Let $D = C \cap L$. From the exact sequence 
$$0 \to \E_{|L}(-D)(-1) \to \E_{|X_{n-1}}(-1) \to \E_{|C}(-1) \to 0$$
we find that $H^1(\E_{|C}(-1)) = 0$. Moreover there is an effective divisor $D'$ on $C$ such that $H_n \cap C = D + D'$
and the exact sequence 
$$0 \to \E_{|C}(-1) \to \E_{|C}(-D) \to \E_{|D'}(-D) \to 0$$
implies that $H^1(\E_{|C}(-D)) = 0$. But then the exact sequence 
$$0 \to \E_{|C}(-D) \to \E_{|X_{n-1}} \to \E_{|L} \to 0$$
gives that the map 
$$H^0(\E_{|X_{n-1}}) \to H^0(\E_{|L})$$
is surjective and we deduce that the map 
$$H^0(\E) \to H^0(\E_{|L})$$
is also surjective.

Consider next the inclusion $\P(\E_{|L}) \subseteq \P(\E)$ so that $(\varphi_{\E})_{|\P(\E_{|L})} = \varphi_{\E_{|L}}$.
Let
$$\E_{|L} \cong \O_{\P^1}(a_1) \oplus \ldots \oplus \O_{\P^1}(a_r)$$ 
with $a_i \ge 0$ for every $1 \le i \le r$ (since $\E_{|L}$ is globally generated). Then the map
$$\lambda_{\E_{|L}} : \Lambda^r H^0(\E_{|L}) \to H^0(\det (\E_{|L}))$$
is surjective. We have diagram \eqref{muk} applied to $\E_{|L}$:
$$\xymatrix{L \ar[d]^{\varphi_{\det (\E_{|L})}} \ar[r]^{\hskip -1.4cm \Phi_{\E_{|L}}} & {\mathbb G}(r-1,h^0(\E_{|L})-1) \ar@{^{(}->}[d]^{P_{\E_{|L}}} \\ \P H^0(\det (\E_{|L})) \ar@{^{(}->}[r] & \P \Lambda^r H^0(\E_{|L})}.$$
Now if there is an $i \in \{1,\ldots,r\}$ with $a_i > 0$, then $\varphi_{\det (\E_{|L})}$ is an isomorphism on its image, and therefore so is $\Phi_{\E_{|L}}$. Therefore also $(\Phi_{\E})_{|L}$ is an isomorphism onto its image. 

On the other hand if $a_i=0$ for every $1 \le i \le r$, it follows that $\varphi_{\det (\E_{|L})}$ maps $L$ to a point, and then so do $\Phi_{\E_{|L}}$ and $(\Phi_{\E})_{|L}$.
\end{proof}
\begin{lemma}
\label{fibre}
Let $X  \subseteq \P^N$ be a smooth irreducible variety of dimension $n \ge 1$. Let $\E$ be a rank $r$ Ulrich vector bundle on $X$. Let $F$ be a scheme-theoretic fiber of $\Phi_{\E}$. Let $Z \subset F$ be a $0$-dimensional subscheme of length $2$ and let $L \subset \P^N$ be the line generated by $Z$. Then $L \subseteq F$. In particular $F$ is a linear space in $\P^N$. 
\end{lemma}
\begin{proof}
By Lemma \ref{span} we have that if $ L \not\subset X$ then $H^0(\E) \to H^0(\E_{|Z})$ is surjective, hence
$$h^0(\I_{Z/X}\otimes \E) = h^0(\E)-2r \le h^0(\E)-r-1$$
and  $Z$ is separated by $\Phi_{\E}$ (see for instance \cite[Prop.~2.4]{ar}), a contradiction. Therefore $L \subseteq X$ and $Z$ is not separated by $(\Phi_{\E})_{|L}$. Thus $(\Phi_{\E})_{|L}$ is constant by Lemma \ref{rette} and then $L \subseteq F$. Since this happens for every two distinct points in $F_{red}$ (if there are) we find that $F_{red}$ is a linear space in $\P^N$. Moreover, by the same reason, $F$ must be smooth, hence $F$ is a linear space in $\P^N$. 
\end{proof}
\section{Proof of the main results}

We first prove Theorem \ref{main1}.


\renewcommand{\proofname}{Proof of Theorem \ref{main1}}

\begin{proof}
Obviously (i) implies (ii). If (ii) holds and $L \subset X$ is a line, then 
\begin{equation}
\label{taut}
\O_{\P(\E_{|L})}(1) = {\O_{\P(\E)}(1)}_{|\P(\E_{|L})}
\end{equation}
is ample, hence $\E_{|L}$ is ample. Thus we get (iii). Now assume (iii) and let $Z \subset \P(\E)$ be a $0$-dimensional subscheme of length $2$. We will prove that 
\begin{equation}
\label{jet}
r_Z : H^0(\O_{\P(\E)}(1)) \to H^0(\O_{\P(\E)}(1)_{|Z}) \ \hbox{is surjective}.
\end{equation}
If $Z \subset \P(\E_x)$ for some $x \in X$ then we have the following commutative diagram
$$\xymatrix{H^0(\O_{\P(\E)}(1)) \ar[d]^f \ar[r]^{\hskip -.3cm r_Z} & H^0(\O_{\P(\E)}(1)_{|Z}) \ar[d]^{\cong} \\ H^0(\O_{\P(\E_x)}(1)) \ar[r]^{\hskip -.3cm r_{Z,x}} & H^0(\O_{\P(\E_x)}(1)_{|Z})}$$
where $f$ is surjective (see Remark \ref{inj}) and $r_{Z,x}$ is surjective because $\O_{\P(\E_x)}(1) \cong \O_{\P^{r-1}}(1)$ is very ample. Therefore $r_Z$ is surjective in this case. 

Hence we can assume that $Z$ is not contained in any fiber of $\pi$ and let $\pi(Z) \subset X$ be the corresponding $0$-dimensional subscheme of length $2$. 
Let $L$ be the line in $\P^N$ generated by $\pi(Z)$. If $L \not\subset X$, which holds in particular if $X$ does not contain lines, then we know by Lemma \ref{span} that $H^0(\E) \to H^0(\E_{|\pi(Z)})$ is surjective, hence so is $r_Z$. If $L \subset X$ it follows by Lemma \ref{rette} that $H^0(\E) \to H^0(\E_{|L})$ is surjective. On the other hand $\E_{|L}$ is ample, hence very ample and therefore $H^0(\E_{|L}) \to H^0(\E_{|\pi(Z)})$ is surjective. But then again $H^0(\E) \to H^0(\E_{|\pi(Z)})$ is surjective, hence so is $r_Z$. Thus (i) is proved.

Next, if (iv) holds, then diagram \eqref{muk} gives that  $|\Im \lambda_{\E}|$ is very ample, hence also (v) holds. Obviously (v) implies (vi). If (vi) holds and $L \subset X$ is a line, then $\det(\E_{|L}) = (\det \E)_{|L}$ is ample, hence $\E_{|L}$ is not trivial. Thus we get (vii). Finally assume (vii).  Let $Z \subset X$ be a $0$-dimensional subscheme of length $2$. We will prove that 
\begin{equation}
\label{jet2}
h_Z : \Im \lambda_{\E} \to H^0((\det \E)_{|Z}) \cong \C^2 \ \hbox{is surjective}.
\end{equation}
This gives that $|\Im \lambda_{\E}|$ is very ample and, together with diagram \eqref{muk}, implies (iv). Let $L$ be the line in $\P^N$ generated by $Z$. If $L \not\subset X$, which holds in particular if $X$ does not contain lines, then we know by Lemma \ref{span} that $g_Z : H^0(\E) \to H^0(\E_{|Z}) \cong \C^{2r}$ is surjective, where $r$ is the rank of $\E$. Consider $Z$ as two possibly infinitely near points $x, x'$. The surjectivity of $g_Z$ implies that we can find, for every $1 \le i \le r$, sections $\sigma_i \in H^0(\E)$ such that $\sigma_i(x)=0$ and $\{\sigma_1(x'),\ldots,\sigma_r(x')\}$ are linearly independent (and similarly swapping $x$ and $x'$). Now this gives that 
$${\sigma_1}_{|Z}(x) \wedge \ldots \wedge {\sigma_r}_{|Z}(x)=0, {\sigma_1}_{|Z}(x') \wedge \ldots \wedge {\sigma_r}_{|Z}(x') \ne 0$$
and similarly swapping $x$ and $x'$. Thus $h_Z$ is surjective.

If $L \subset X$ it follows by Lemma \ref{rette} that $H^0(\E) \to H^0(\E_{|L})$ is surjective and that, being $\E_{|L}$ globally generated, that $\lambda_{\E_{|L}}$ is surjective. This gives the following commutative diagram
$$\xymatrix{\Lambda^r H^0(\E) \ar@{->>}[d] \ar[r]^{\lambda_{\E}} & \Im \lambda_{\E} \ar[d]^{h_L}\ar@{^{(}->}[r] & H^0(\det \E) \ar[dl] \\ \Lambda^r H^0(\E_{|L}) \ar@{->>}[r]^{\hskip -.2cm \lambda_{\E_{|L}}} & H^0(\det(\E_{|L})) & }$$
hence $h_L$ is surjective. Now $\det(\E_{|L})$ is very ample hence 
$$H^0(\det(\E_{|L})) \to H^0(\det(\E_{|L})_{|Z}) = H^0((\det \E)_{|Z})$$
is surjective and composing with $h_L$ we get that $h_Z$ is surjective. Thus \eqref{jet2} is proved and (iv) holds. Alternatively the fact that (iv) holds assuming (vii) can be proved as follows. By a standard separation criterion (see for instance \cite[Prop.~2.4]{ar}), to get that $\Phi_{\E}$ is an embedding one just needs to prove that 
\begin{equation}
\label{stima}
h^0(\I_{Z/X}\otimes \E) \le h^0(\E)-r-1.
\end{equation}
This clearly holds when the line $L$ generated by $Z$ is not contained in $X$, since Lemma \ref{span} gives that $h^0(\I_{Z/X}\otimes \E) = h^0(\E)-2r$. On the other hand, if $L \subset X$, we know by Lemma \ref{rette} that $H^0(\E) \to H^0(\E_{|L})$ is surjective. Therefore $h^0(\I_{Z/X}\otimes \E) = h^0(\E)-\rk \psi$, where $\psi: H^0(\E_{|L}) \to H^0(\E_{|Z})$. Now $\E_{|L}$ is globally generated and non trivial, hence we can write $\E_{|L} \cong \O_{\P^1}(a_1) \oplus \ldots \oplus \O_{\P^1}(a_r)$ with $a_1 \ge \ldots a_r \ge 0$ and $a_1 \ge 1$. Since $\Ker \psi = H^0(\E_{|L}(-Z))$, letting $k$ be the number of $i \in \{1,\ldots,r\}$ such that $a_i \ge 2$, we get 
$$\rk \psi = \sum\limits_{i=1}^r(a_i+1)-\sum\limits_{i=1}^k(a_i-1)=\sum\limits_{i=k+1}^ra_i+r+k \ge r+1$$
hence \eqref{stima} holds.
\end{proof}

\begin{remark}
With the same method of proof of Theorem \ref{main1} one can show that if $\E$ is a rank $r$ ample Ulrich vector bundle on $X \subset \P^N$ then, the Seshadri constants, on any $x \in X$, satisfy:
\begin{itemize}
\item [(i)] $\varepsilon(\E,x) \ge 1$;
\item [(ii)] $\varepsilon(\det \E,x) \ge r$.
\end{itemize}
\end{remark}
The same can be proven using Theorem \ref{main1},  \cite{bss} and \cite{fm}.

Next we prove Theorem \ref{main2}.

\renewcommand{\proofname}{Proof of Theorem \ref{main2}}

\begin{proof}  
It follows by Lemma \ref{fibre} and \eqref{pc} that any scheme-theoretic fiber of $\Phi_{\E}$ is a linear space of dimension at least $\phi(\E)=n-\nu(\det \E)$. Now let $F$ be any scheme-theoretic fiber of $\varphi_{\E}$, so that $\dim F \ge n+r-1-\nu(\E)$. We want to show that the scheme-theoretic image $\pi(F) \subseteq X \subseteq \P^N$ is a linear space. To this end it is enough to show, as done in Lemma \ref{fibre}, that for any $0$-dimensional subscheme $Z \subset \pi(F)$ of length $2$, then the line $L:= \langle Z \rangle$ is contained in $\pi(F)$. By the second part of Remark \ref{inj} there is a $0$-dimensional subscheme $Z' \subset F$ of length $2$ such that $Z = \pi(Z')$. By hypothesis we have that $H^0(\O_{\P(\E)}(1)) \to H^0(\O_{\P(\E)}(1)_{|Z'})$ is not surjective, that is $H^0(\E) \to H^0(\E_{|Z})$ is not surjective. Hence Lemma \ref{span} gives that $L \subset X$. Also Lemma \ref{rette} implies that also $H^0(\E_{|L}) \to H^0(\E_{|Z})$ is not surjective. Thus $\E_{|L}$ must have a maximal trivial direct summand, say $\O_{\P^1}^{\oplus k}$ for some integer $1 \le k \le r$ and $\P(\E_{|L})$ gets mapped by $\varphi_{\E}$ onto a rational normal scroll that is a cone with vertex $V = \P^{k-1} \subset \P^{rd-1}$ when $k \le r-1$ or onto $V= \P^{r-1}$ when $k=r$. It follows that $\varphi_{\E}(F) \subseteq V$. For any $x'' \in L $ we have that $V \subseteq \varphi_{\E}(\pi^{-1}(x''))$, hence there exists $z'' \in \pi^{-1}(x'')$ such that $\varphi_{\E}(z'') = \varphi_{\E}(z)$. Therefore $z'' \in F$ and $x'' = \pi(z'') \in \pi(F)$. 
\end{proof}

\begin{remark}
\label{kgg}
Using Theorem \ref{main2} and \cite[Rmk.~1]{bss} we get that if $\E$ is an Ulrich vector bundle on $X \subseteq \P^N$ and $H$ is a hyperplane divisor, then $K_X+tH$ is globally generated provided that
$$t \ge \max\{n - \nu(\det \E)+1, \nu(\det \E)+1\}.$$
\end{remark}


\begin{remark}  
\label{delp}
Let $X \subset \P^7$ be the Del Pezzo surface of degree $7$, so that $X$ is the blow-up of the plane in two points. Let $E_1, E_2$ be the exceptional divisors and let $\widetilde H$ be the inverse image of a line in the plane. Let $\E=\O_X(H+E_1-E_2) = \O_X(3\widetilde H-2E_2)$. Then $\E$ is an Ulrich line bundle on $X$ by \cite[Prop.~4.1(i)]{b1}. Now $\Phi_{\E}$ is birational onto its image and behaves differently on the three lines contained in $X$: $E_1$ is contracted, $E_2$ maps to a conic and $\widetilde H-E_1-E_2$ maps to a line.
\end{remark}

Now we prove the corollaries.



\renewcommand{\proofname}{Proof of Corollary \ref{cor1}}

\begin{proof}
Obviuously (i) implies (ii) and (iii) implies (iv) by \eqref{muk}. Vice versa assume that $\E$ (respectively $\det \E$) is big. Then $\nu(\E)=n+r-1$ (respectively $\nu(\det \E)=n$). Hence the general fiber of $\varphi_{\E}$ (respectively of $\Phi_{\E}$) is $0$-dimensional, so it is a point by Theorem \ref{main2}. Therefore $\varphi_{\E}$ (respectively $\Phi_{\E}$) is birational onto its image. 
\end{proof}
\renewcommand{\proofname}{Proof of Corollary \ref{cor2}}

\begin{proof}
Set $\nu = \nu(\E)$. Since $\dim \Im \varphi = \nu$ we have that $\dim \varphi^{-1}(y) \ge n+r-\nu-1$ for any $y \in \Im \varphi$.

Let $x \in X$ be a general point and let
$$Z = \{z \in X : P_z \cap P_x \ne \emptyset\}.$$
We claim that $\dim Z \ge n+r-\nu$.

To see this consider the incidence correspondence
$$\I = \{ (y, z) \in P_x \times X : y \in P_z \} \subset P_x \times X$$
together with its projections 
\[ \xymatrix{& \I \ \ \ar[dl]_{\pi_1} \ar[dr]^{\pi_2} & \\ \ \ \ \ \ P_x & & \ X \ \ \ .} \]
Let $y \in P_x$. Then $(y, x) \in \I$, hence $\pi_1$ is surjective. Next observe that $\dim \pi_1^{-1}(y) \ge n+r-\nu-1$. 

In fact we know that $\dim \varphi^{-1}(y) \ge n+r-\nu-1$. Since $\pi$ is injective on $\varphi^{-1}(y)$ by Remark \ref{inj}, we get that $\dim \pi(\varphi^{-1}(y)) \ge n+r-\nu-1$. Also 
$$\pi(\varphi^{-1}(y)) \subseteq \{z \in X : y \in P_z\} \cong \pi_1^{-1}(y):$$ 
if $z \in \pi(\varphi^{-1}(y))$ then there is $u \in \varphi^{-1}(y)$ such that $z = \pi(u)$. But $u \in \P(\E_z)$, hence $y = \varphi(u) \in P_z$.

Therefore $\dim \pi_1^{-1}(y) \ge n+r-\nu-1$ and then $\dim \I \ge n+2r-\nu-2$. Pick an irreducible component $\I_1 \subseteq \I$ such that $\dim \I_1 \ge n+2r-\nu-2$. Let $z \in \pi_2(\I_1)$. Then $\pi_2^{-1}(z) \cong P_x \cap P_z$ hence $\dim \pi_2^{-1}(z) \le r-1$, therefore $\pi_2(\I_1)$ is irreducible and $\dim \pi_2(\I_1) \ge n+r-\nu-1 \ge 1$ because $\E$ is not big. Now let $z \in \pi_2(\I_1)$ be general. Then $z \ne x$. If $\dim \pi_2^{-1}(z) = r-1$ then $P_x=P_z$ hence $\Phi(x)=\Phi(z)$. But $x \in X$ is a general point and then Theorem \ref{main2} implies that $x=z$, a contradiction. Hence $\dim \pi_2^{-1}(z) \le r-2$ and this gives that $\dim \pi_2(\I_1) \ge n+r-\nu$. On the other hand  $\pi_2(\I_1) \subseteq Z$, so that $\dim Z \ge n+r-\nu$.

Given any $z \in Z, z \ne x$ it follows by Theorem \ref{main2} that $\pi(\varphi^{-1}(y)) \subset X$ is a linear subspace of dimension at least $n+r-\nu-1$ in $\P^N$ for any $y \in P_z\cap P_x$. Moreover 
$$Z = \bigcup\limits_{z \in Z \setminus \{x\}} \bigcup\limits_{y \in P_z\cap P_x} \pi(\varphi^{-1}(y)).$$ 
As $\dim Z \ge n+r-\nu$ we find that through a general point of $X$ there are infinitely many $(n+r-1-\nu)$-dimensional linear subspaces contained in $X$. Since this is a closed condition, the same holds for any $x\in X$.
\end{proof}

\renewcommand{\proofname}{Proof of Corollary \ref{cor3}}

\begin{proof}
As above we set $\Phi=\Phi_{\E}, \varphi=\varphi_{\E}, \phi=\phi(\E)$ and $\nu=\nu(\E)$. Note that $c_1(\E)^n=0$ is equivalent to $\phi \ge 1$ by \eqref{pc}. Let $F$ be a general fiber of  $\Phi$, so that we know by Theorem \ref{main2} that $F=\P^{\phi} \subseteq \P^N$ and that $X$ is covered by a family of linear $\P^{\phi}$'s. Let $\Phi(F)=[P] \in {\mathbb G}(r-1,rd-1)$. For a general $y \in \Im \varphi$ set
$$W_y = \{x \in X : y \in P_x\}.$$
Note that $\pi_{|\varphi^{-1}(y)}: \varphi^{-1}(y) \to W_y$ is an isomorphism by Remark \ref{inj}, hence
$$\dim W_y = \dim \varphi^{-1}(y) = n+r-1-\nu.$$
We will now prove that $F = W_y$ for any $y \in P$. This gives the corollary.

First $F \subseteq W_y$. In fact if $x \in F$ then $[P] = \Phi(F)=\Phi(x)=[P_x]$ so that $y \in P = P_x$, hence $x \in W_y$.

Now assume that there exists $x \in W_y \setminus F$. For any $x' \in F$ we have that $\langle x,x' \rangle \subset X$. In fact if $\langle x,x' \rangle$ is not contained in $X$, then Lemma \ref{span} implies that $P_x \cap P_{x'} = \emptyset$. On the other hand $[P] = \Phi(F)=\Phi(x')=[P_{x'}]$, hence $y \in P=P_{x'}$. But $x \in W_y$, hence $y \in P_x$, a contradiction.

Therefore $\P^{\phi+1} = \langle x, F \rangle \subset X$. Now the restriction of $\Phi$ to $\P^{\phi+1} = \langle x, F \rangle$ is a morphism that contracts a hyperplane $F=\P^\phi$ to a point, hence it is constant. But this gives the contradiction $x \in F$.

Therefore $F = W_y$ and the corollary is proved.
\end{proof}

%
%

\renewcommand{\proofname}{Proof of Corollary \ref{cor4}}
\begin{proof}
If $(X,\O_X(1), \E)$ is as in (i) or (ii) it follows by \cite[Lemma 4.1]{lo} that $\E$ is Ulrich and obviously $c_1(\E)^{\lfloor \frac{n}{2} \rfloor +1} =0$.

Assume now that $\E$ is Ulrich with $c_1(\E)^{\lfloor \frac{n}{2} \rfloor +1} =0$. 

%
%


Since $\lfloor \frac{n}{2} \rfloor +1 \le n$ we have that $(\det \E)^n = 0$. If $\rho(X)=1$, then $\det\E \cong \O_X$  and this gives that $\E \cong \O_X^{\oplus r}$. Hence $rd=h^0(\E)=r$ and we are in case (i). 

Therefore we may assume that $\rho(X)\geq 2$.

Since $\nu(\det \E) \le \frac{n}{2}$, we deduce from \eqref{pc} that $\phi(\E) \ge \frac{n}{2}$. By Theorem \ref{main2} we know that $X$ is covered by a family of $\P^{\phi(\E)}$'s, hence \cite[Main Thm.]{sa2} implies that there is a smooth irreducible projective variety $B$ of dimension $b \ge 1$, a very ample rank $n-b+1$ vector bundle $\F$ on $B$ such that $(X, \O_X(1)) \cong (\P(\F), \O_{\P(\F)}(1))$. Let $p : \P(\F) \to B$ be the projection. Then \cite[Main Thm.]{sa2} gives that there is a nonempty open subset $U$ of $X$ such that 
$$\P^{\phi(\E)} = \Phi^{-1}(\Phi(u)) \subseteq p^{-1}(p(u))= \P^{n-b} \ \hbox{for every} \ u \in U.$$ 
In particular we have that $\phi(\E) \le n-b$, hence $b \le n-\phi(\E) \le \frac{n}{2}$. 

We actually claim that all fibers of $\Phi$ and $p$ coincide. 

Let $x \in X$. Then $p^{-1}(p(x))=\P^{n-b}$ and let us show that $\Phi(p^{-1}(p(x)))$ is a point. In fact, if not, then 
$$\frac{n}{2} \le \phi(\E) \le n-b = \dim \Phi(p^{-1}(p(x))) \le \dim \Phi(X) = n-\phi(\E) \le \frac{n}{2}$$ 
hence $\Phi(p^{-1}(p(x)))=\Phi(X)$. Now for every $u \in U$ there exists $x' \in p^{-1}(p(x))$ such that $\Phi(x')=\Phi(u)$, so that $x' \in \Phi^{-1}(\Phi(u)) \subseteq p^{-1}(p(u))$. This gives that $p(u)=p(x')=p(x)$, hence $u \in p^{-1}(p(x))$. But then we get the contradiction $U \subset p^{-1}(p(x))$.

Therefore $\Phi(p^{-1}(p(x)))$ is a point and 
$$p^{-1}(p(x)) \subseteq \Phi^{-1}(\Phi(x))$$
holds for every $x \in X$. Let us see that they are actually equal. Suppose that there exists $x' \in\Phi^{-1}(\Phi(x)) \setminus p^{-1}(p(x))$. Now Theorem \ref{main2} gives that $\Phi^{-1}(\Phi(x))=\P^s$ for some $s \le n-1$. Moreover $p^{-1}(p(x')) \subseteq \Phi^{-1}(\Phi(x'))=\Phi^{-1}(\Phi(x))$, hence we have a $\P^s$ containing two disjoint $\P^{n-b}$'s, namely $p^{-1}(p(x'))$ and $p^{-1}(p(x))$. But this is a contradiction since $b \le \frac{n}{2}$.

Thus $p^{-1}(p(x)) = \Phi^{-1}(\Phi(x))$ for every $x \in X$. 

Consider the Stein factorization 
$$\xymatrix{X \ar[dr]_{\Phi} \ar[r]^{\widetilde \Phi} & B_1 \ar[d]^g \\ & \Phi(X)}$$
so that $\widetilde \Phi_*\O_X \cong \O_{B_1}$ and $g$ is finite. Since the fibers of $\Phi$ are connected by Theorem \ref{main2}, it follows that $g$ is bijective and therefore $p^{-1}(p(x)) = \Phi^{-1}(\Phi(x)) = \widetilde \Phi^{-1}(\widetilde \Phi(x))$ for every $x \in X$. Hence we can apply \cite[Lemma 1.15(b)]{de} and deduce that $\widetilde \Phi$ factorizes through $p$ and that $p$ factorizes through $\widetilde \Phi$. Since they have the same fibers it follows that $B \cong B_1$. 


Now we know that $\E = \Phi^*\U = \widetilde \Phi^*(g^*\U)$, where $\U$ is the tautological bundle on $\mathbb G(r-1, \P H^0(\E))$. Thus there is a rank $r$ vector bundle $\H$ on $B$ such that $\E \cong p^*\H$. Setting $\G = \H(-\det \F)$ and using \cite[Lemma 4.1]{lo} we deduce that $(X,\O_X(1), \E)$ is as in (ii).
\end{proof}
\begin{remark}
\label{sha}
Let $Y$ be a smooth irreducible variety of dimension $m$, let $L$ be a very ample line bundle on $Y$. Let $X$ be a hyperplane section of the Segre embedding of $\P^{m-1} \times Y$. As shown in \cite[Ex.~14.1.5]{bs} the restriction of the second projection $p : X \to Y$ exhibits $(X,H)$ as a scroll over $Y$ with general fiber a linear $\P^{m-2}$, but having some fibers equal to a linear $\P^{m-1}$. Hence $X$ is of dimension $n=2m-2$ and, if $(Y,L)$ is not covered by lines, $X$ is not a $\P^{n-b}$-bundle over a variety of dimension $b$. On the other hand, choosing an Ulrich vector bundle $\G$ for $(Y,L)$ and setting $\E = p^*(\G((m-1))L)$ we get by \cite[Lemma 4.1]{lo} an Ulrich vector bundle on $X$ such that $c_1(\E)^{\lfloor \frac{n}{2} \rfloor+1} \ne 0$ and $c_1(\E)^{\lfloor \frac{n}{2} \rfloor+2} = 0$.
\end{remark}

\renewcommand{\proofname}{Proof of Corollary \ref{cor5}}
\begin{proof}
If $(X,\O_X(1), \E)$ is as in (i) or (ii) it follows by \cite[Lemma 4.1]{lo} that $\E$ is Ulrich and obviously $c_1(\E)^{n-b+1} =0$.

Vice versa assume that $\E$ is Ulrich with $c_1(\E)^{n-b+1} =0$.


If $n=b$ then $c_1(\E)=0$, so $\Phi_{\E}$ is a constant map and $\E \cong \O_X^{\oplus r}$. Hence $rd=h^0(\E)=r$ and we get (i). From now on assume that $n \ge b+1$.

By hypothesis we have that that $\nu(\det \E) \le n-b$, hence \eqref{pc} gives that
$$\phi(\E) \ge b.$$
Set $\Phi = \Phi_{\E}$ and, for any $x \in X$, set $f_x = p^{-1}(p(x))$ and $F_x = \Phi^{-1}(\Phi(x))$. Hence $f_x$ is a linear subspace of $\P^N$ of dimension $n-b$, while $F_x$ is a linear subspace of $\P^N$ by Theorem \ref{main2} with $\dim F_x \ge \phi(\E)$  and equality holds for a general $x \in X$. Moreover $x \in f_x \cap F_x$.

\begin{claim}
\label{dicot}
Either $f_x \cap F_x = \{x\}$ or $f_x = F_x$.
\end{claim}
\renewcommand{\proofname}{Proof}
\begin{proof}
Assume that $f_x \cap F_x \ne \{x\}$, hence $f_x \cap F_x$ is a positive dimensional linear subspace of $\P^N$. Now $p_{|F_x} : F_x \to B$ contracts $f_x \cap F_x$ to a point. Since any morphism from a projective space is either constant or finite, we get that $p(F_x)$ is a point and therefore $F_x \subseteq f_x$. Similarly $\Phi_{|f_x}: f_x \to \Phi(X)$ contracts $f_x \cap F_x$ to a point, hence, as above, $\Phi(f_x)$ is a point and therefore $f_x \subseteq F_x$. This proves the Claim.
\end{proof}
Morover we have
\begin{claim}
\label{dicot2}
If $f_x \ne F_x$, then $\phi(\E)=b, F_x = \P^b \cong B$. Also $\Phi_{|f_x}: f_x \to \Phi(X)$ is an embedding.
\end{claim}
\renewcommand{\proofname}{Proof}
\begin{proof}
Let $z \in F_x$, so that $F_z=F_x$. If $f_z \cap F_z \ne \{z\}$ then $F_x = F_z = f_z$ by Claim \ref{dicot}. Hence $x \in f_z$ and therefore $f_x = f_z = F_x$, a contradiction. Therefore $\{z\} = f_z \cap F_z = f_z \cap F_x$, hence $p_{|F_x}: F_x \to B$ is injective. The injectivity of $\Phi_{|f_x}: f_x \to \Phi(X)$ is proved in the same way. As in the proof of Theorem \ref{main1}, if $\Phi_{|f_x}$ were not an embedding, it would have a positive dimensional fiber, a contradiction. Moreover
$$b \le \phi(\E) \le \dim F_x \le \dim B = b$$
so that $\phi(\E)=b, F_x = \P^b$ by Theorem \ref{main2} and $B \cong \P^b$. This proves the Claim.
\end{proof}
Consider the Stein factorization 
$$\xymatrix{X \ar[dr]_{\Phi} \ar[r]^{\widetilde \Phi} & B_1 \ar[d]^g \\ & \Phi(X)}$$
and set $\widetilde F_x = \widetilde \Phi^{-1}(\widetilde \Phi(x))$. As in the proof of Corollary \ref{cor4} we see that $F_x = \widetilde F_x$ for every $x \in X$.

We divide the proof in two cases.

{\bf Case 1:} There is a point $x \in X$ such that $f_x = F_x$. 


We show that $f_{x'}=F_{x'}$ for every $x'\in X$. 

Assume to the contrary that there is $x' \in X, x' \ne x$ such that $f_{x'} \ne F_{x'}$. Then $f_{x'} \cap F_{x'}=\{x'\}$ by Claim \ref{dicot} and $\dim F_{x'}=b$ by Claim \ref{dicot2}. Therefore $f_x \cdot  F_{x'} = f_{x'} \cdot F_{x'}=1$, hence $F_x \cap F_{x'} = f_x \cap F_{x'}\neq\emptyset$ and we deduce that $F_x=F_{x'}$. But then
$$f_x\cap f_{x'}=F_{x'}\cap f_{x'}=\{x'\}$$ 
hence $f_x = f_{x'}$, giving the contradiction $\P^{n-b}=f_x = \{x'\}$. 

Therefore $f_{x'}=F_{x'}$ for every $x'\in X$ and this gives that $\phi(\E) =n-b$ and $b \le \frac{n}{2}$ since $\phi(\E) \ge b$. Now, precisely as in the proof of Corollary \ref{cor4}, we can apply \cite[Lemma 1.15(b)]{de} and deduce that we are in case (ii).

This concludes Case 1.
 
{\bf Case 2:}  $f_x \ne F_x$ for every $x \in X$. 

Then Claim \ref{dicot2} gives that $\phi(\E) = b$ and, for every $x \in X$, $F_x = \P^b \cong B$. Also Claim \ref{dicot2} gives that $\Phi_{|f_x}: f_x \to \Phi(X)$ is an embedding and since $\dim \Phi(X) =  n-b$ we get that $\Phi(X) \cong \P^{n-b}$. 


Since $F_x$ is a linear space of dimension $b$ for every $x\in X$, we get that $\Phi:X\to\P^{n-b}$ is a linear $\P^b$-bundle.
Now $f_x \ne F_x$ for every $x \in X$, hence the two linear bundle structures given by $p : X \to B \cong \P^b$ and $\Phi : X \to \Phi(X) \cong \P^{n-b}$ are different, that is that there is no isomorphism $h : \Phi(X) \to B$ such that $p= h \circ \Phi$. Then it follows by \cite[Thm.~A]{sa1} that $(X, H) \cong (\P^b \times \P^{n-b}, \O_{\P^b}(1) \boxtimes \O_{\P^{n-b}}(1))$. Since $\E = \Phi^*\U$ where $\U$ is the tautological bundle on $\mathbb G(r-1, \P H^0(\E))$, \cite[Lemma 4.1]{lo} together with \cite[Prop.~2.1]{es} (or \cite[Thm.~2.3]{b1}) give that $\E \cong q^*(\O_{\P^{n-b}}(b))^{\oplus r}$, where $q$ is the second projection. Then we are in case (i).
This concludes Case 2.
\end{proof}
\renewcommand{\proofname}{Proof}

With the same methods we can also classify Ulrich vector bundles with non big determinant on the Segre product. 


\begin{cor}
\label{prodotto}
Let $\E$ be a rank $r$ vector bundle on $\P^b \times \P^{n-b} \subset \P^N$ embedded with $\O_{\P^b}(1) \boxtimes \O_{\P^{n-b}}(1)$. Then $\E$ is Ulrich with $\det\E$ not big if and only if $\E$ is either $p^*(\O_{\P^{b}}(n-b)^{\oplus r})$ or $q^*(\O_{\P^{n-b}}(b)^{\oplus r})$ where $p$ and $q$ are the projections.
\end{cor}

\begin{proof} 
If $\E$ is either $p^*(\O_{\P^{b}}(n-b)^{\oplus r})$ or $q^*(\O_{\P^{n-b}}(b)^{\oplus r})$ then it is Ulrich by  \cite[Lemma 4.1]{lo} and clearly $\det\E$ is not big.

Vice versa assume that $\E$ is Ulrich with $\det\E$ not big. 

Set $\Phi = \Phi_{\E}$. For any $x\in X=\P^b\times\P^{n-b}$, let $f_x=p^{-1}(p(x))$, $g_x=q^{-1}(q(x))$ and $F_x=\Phi^{-1}(\Phi(x))$. We know by Theorem \ref{main2} that $F_x$ is a linear subspace contained in $X$ and passing through $x$. But it is well known that, on the Segre embedding, any linear subspace must be contained in a fiber. Then either $F_x\subseteq f_x$ or $F_x\subseteq g_x$. Assume for example that $F_x\subseteq f_x$ for some $x\in X$. Since $\det\E$ is not big we get $\dim F_x\geq 1$ by \eqref{pc}. Moreover $\Phi(F_x)=\Phi(x)$, hence, as in the proof of Claim \ref{dicot}, $\Phi(f_x)=\Phi(x)$ and therefore $F_x=f_x$. We prove now that $F_{x'}=f_{x'}$ for any $x'\in X$. In fact, otherwise, $F_{x'}=g_{x'}$ for some $x\neq x'\in X$ and we get that $F_x \cap F_{x'} = f_x \cap g_{x'} = \{(p(x),q(x'))\}$. But then $F_x = F_{x'}$ giving the contradiction $F_x = \{(p(x),q(x'))\}$.

Then we get $\E=q^*(\O_{\P^{n-b}}(b)^{\oplus r})$ as in the proof of Corollary \ref{cor5}. The same works if $F_x\subseteq g_x$ for some $x\in X$.

Alternatively we can prove the corollary as follows. 

Set $L = p^*\O_{\P^b}(1)$ and $M= q^*\O_{\P^{n-b}}(1)$. Then we can write $\det \E = aL+cM$ for some $a, c \in \ZZ$ and
$$0 = (aL+cM)^n=\binom{n}{b}a^bc^{n-b}$$
implies that either $a=0$ or $c=0$. It follows by \cite[Lemmas 5.1 and 4.1]{lo} that $\E$ is as claimed.
\end{proof}


\begin{thebibliography}{CMRPL}

\bibitem[A]{ar} E.~Arrondo.
\textit{Subvarieties of Grassmannians}. 
Notes available at http://www.mat.ucm.es/$\sim$arrondo/trento.pdf


\bibitem[Be]{b1} A.~Beauville.
\textit{An introduction to Ulrich bundles}. 
Eur. J. Math. \textbf{4} (2018), no. 1, 26-36.


\bibitem[Br]{br} D.~Brotbek.
\textit{Hyperbolicity related problems for complete intersection varieties}. 
Compos. Math. \textbf{150} (2014), no. 3, 369-395.

\bibitem[BS]{bs} M.~C.~Beltrametti, A.~J.~Sommese.
\textit{The adjunction theory of complex projective varieties}. 
De Gruyter Expositions in Mathematics, \textbf{16}. Walter de Gruyter \& Co., Berlin, 1995.

\bibitem[BSS]{bss} M.~C.~Beltrametti, M.~Schneider, A.~J.~Sommese.
\textit{Chern inequalities and spannedness of adjoint bundles}.
Proceedings of the Hirzebruch 65 Conference on Algebraic Geometry (Ramat Gan, 1993), 97-107,
Israel Math. Conf. Proc., \textbf{9}, Bar-Ilan Univ., Ramat Gan, 1996. 

\bibitem[CMRPL]{cmp} L.~Costa, R.~M.~Mir\'o-Roig, J.~Pons-Llopis.
\textit{Ulrich bundles}.
De Gruyter Studies in Mathematics, \textbf{77}, De Gruyter 2021. 

\bibitem[D]{de} O.~Debarre.
\textit{Higher-dimensional algebraic geometry}. 
Universitext. Springer-Verlag, New York, 2001. xiv+233 pp.

\bibitem[ES]{es} D.~Eisenbud, F.-O.~Schreyer.
\textit{Resultants and Chow forms via exterior syzygies. With an appendix by Jerzy Weyman}. 
J. Amer. Math. Soc. \textbf{16} (2003), no. 3, 537-579.

\bibitem[FM]{fm} M.~Fulger, T.~Murayama.
\textit{Seshadri constants for vector bundles}.
J. Pure Appl. Algebra \textbf{225} (2021), no. 4, 106559, 35 pp. 

\bibitem[La]{laz} R.~Lazarsfeld. \textit{Positivity in algebraic geometry, I}.
Ergebnisse der Mathematik und ihrer Grenzgebiete, 3. Folge  \textbf{48}, Springer-Verlag, Berlin, 2004.

\bibitem[Lo]{lo} A.~F.~Lopez. 
\textit{On the positivity of the first Chern class of an Ulrich vector bundle}.
Preprint 2020, arXiv:2008.07313. 
To appear on Commun. Contemp. Math. 



\bibitem[M]{mu} S.~Mukai.
\textit{Vector bundles and Brill-Noether theory}. 
Current topics in complex algebraic geometry (Berkeley, CA, 1992/93), 145-158, Math. Sci. Res. Inst. Publ., \textbf{28}, Cambridge Univ. Press, Cambridge, 1995.


\bibitem[S1]{sa1} E.~Sato.
\textit{Varieties which have two projective space bundle structures}. 
J. Math. Kyoto Univ. \textbf{25} (1985), no. 3, 445-457.

\bibitem[S2]{sa2} E.~Sato.
\textit{Projective manifolds swept out by large-dimensional linear spaces}. 
Tohoku Math. J. (2) \textbf{49} (1997), no. 3, 299-321.
 
\end{thebibliography}
\end{document}